\documentclass{amsart}
\usepackage{amssymb}
\newcommand{\kk}{\mathbb{K}}

\newcommand{\SL}{\mathop{\mathrm{SL}}}

\newtheorem{prop}{Proposition}

\newtheorem{corollary}{Corollary}

\theoremstyle{plain}
\newtheorem{theor}{Theorem}
\newtheorem*{prop*}{Proposition}
\newtheorem*{theor*}{Theorem}
\newtheorem{lem}{Lemma}

\theoremstyle{definition}
\newtheorem{de}{Definition}

\theoremstyle{remark}
\newtheorem{rem}{Remark}

\newcounter{property}

\newcounter{prooperty}

\makeatletter
\def\keywords#1{{\def\@thefnmark{\relax}\@footnotetext{#1}}}

\makeatother

\begin{document}

\title[Affine toric $\SL(2)$\,-\,embeddings]
{Affine toric $\SL(2)$\,-\,embeddings}

\author[S. A. Gaifullin]{Sergey A. Gaifullin}
\address{Department of Higher Algebra, Faculty of Mechanics and Mathematics,
Moscow state University, Leninskie Gory, Moscow, GSP-2, 119992
Russia.} \email{sergaifu@mccme.ru}


\begin{abstract}

In 1973 V.L.Popov classified affine $\mathrm{SL}(2)$-embeddings.
He proved that a locally transitive $\mathrm{SL}(2)$-action on a
normal affine three-dimensional variety $X$ is uniquely
determined by a pair $(\frac{p}{q}, r)$, where $0<\frac{p}{q}\le
1$ is an uncancelled fraction and $r$ is a positive integer. Here
$r$ is the order  of the stabilizer of a generic point. In this
paper we show that the variety $X$ is toric, i.e. admits a locally
transitive action of an algebraic torus, if and only if $r$ is
divisible by $q-p$. To do this we prove the following necessary
and sufficient condition for an affine $G/H$-embedding to be
toric. Suppose $X$ is a normal affine variety, $G$ is a simply
connected semisimple algebraic group acting regularly on $X$, $H$
is a closed subgroup of $G$ such that the character group
$\mathfrak{X}(H)$ is finite and $G/H\hookrightarrow X $ is a
dense open equivariant embedding. Then $X$ is toric if and only
if there exist a quasitorus $\widehat{T}$ and a
$(G\times\widehat{T})$-module $V$ such that $X\stackrel{G}{\cong}
V/\!\!/\widehat{T}$. The key role in the proof plays D.\,Cox's
construction.
\end{abstract}

\maketitle
\section*{Introduction}
We work over an algebraically closed field $\mathbb{K}$ of
characteristic zero.

A normal variety $X$ is called {\it toric} if it admits a locally
transitive action of an algebraic torus $T$. Let us fix the action
of $T$ on $X$. The theory of toric varieties (see, for example,
\cite{ful}) assigns a rational strongly convex polyhedral cone to
an affine toric variety. This cone uniquely determines the
variety up to a $T$-equivariant isomorphism.

It follows from D.\,Cox's paper ~\cite{cox} that an affine toric
variety without non-constant invertible functions can be realized
as a categorical quotient $V/\!\!/\widehat{T}$ of a vector space
$V$ by a linear action of a quasitorus $\widehat{T}$. This vector
space is the spectrum of the {\it Cox ring}. On the other hand,
it is not difficult to show that a categorical quotient of a
vector space by a linear action of a quasitorus is toric.

Further all varieties are normal, affine and irreducible, and all
actions are regular.

Suppose $X$ is a variety, $G$ is a simply connected semisimple
group acting on $X$, and $H$ is a closed subgroup of $G$ such
that the character group $\mathfrak{X}(H)$ is finite. Let
$G/H\hookrightarrow X $ be an open equivariant embedding. The
main idea of this paper consists in the following necessary and
sufficient condition for $X$ to be toric.
\begin{prop*}
Affine G-variety $X$ with an open $G$-equivariant embedding
$G/H\hookrightarrow X$ is toric if and only if there exist a
quasitorus $\widehat{T}$ and a $(G\times\widehat{T})$-module $V$
such that $X\stackrel{G}{\cong} V/\!\!/\widehat{T}$.
\end{prop*}
The proof is based on the construction of $(G\times
\widehat{T})$\,-\,action on the spectrum of the Cox ring. It
follows from~\cite{kp} that this action is linear.

In Section \ref{class} we investigate three-dimensional varieties
with a locally transitive $\mathrm{SL}(2)$-action. Recall that a
(normal affine irreducible) three-dimensional variety $X$ with
locally transitive and locally free $\mathrm{SL}(2)$-action is
called an {\it affine $\mathrm{SL}(2)$\,-\,embedding}. Affine
$\mathrm{SL}(2)$\,-\,embeddings are studied in~\cite{p1} (see
also~\cite[Ch.III, $\S$~4]{kr}). In~\cite{p1} it is also proved
that if an $\mathrm{SL}(2)$\,-\,action on a three-dimensional
variety $X$ is locally transitive but not transitive, then the
stabilizer of a generic point is a cyclic group. By $r$ we denote
the order of this group. Then the variety $X$ is called an {\it
affine $\mathrm{SL}(2)/\mathbb{Z}_r$\,-\,embedding}. Let $U$ be a
maximal unipotent subgroup of $\mathrm{SL}(2)$. The algebra of
$U$-invariants of the action is a monomial subalgebra of a
polynomial algebra in two variables. This subalgebra determines a
rational number $0<\frac{p}{q}\le 1$, which is called the {\it
height} of the action. We shall assume that $p$ and $q$ are
relatively prime. One of the main results of V.L.Popov's theory
asserts that affine $\mathrm{SL}(2)/\mathbb{Z}_r$\,-\,embeddings
considered up to an $\mathrm{SL}(2)$\,-\,equivariant isomorphism
are in one-to-one correspondence with pairs $(\frac{p}{q}, r)$,
$0<\frac{p}{q}\le 1$, $r\in\mathbb{N}$. In this paper we find all
pairs $(\frac{p}{q}, r)$ such that the corresponding variety $X$
is toric.

In Section \ref{class} we use the previous proposition for
$G=\mathrm{SL}(2)$, $H=\mathbb{Z}_r$. Further we find out when the
$\mathrm{SL}(2)$\,-\,action on $V/\!\!/ \widehat{T}$ is locally
transitive and determine the corresponding pairs $(\frac pq,r)$.
Finally we obtain the following theorem.
\begin{theor*}
Suppose $X$ is a three-dimensional normal affine irreducible
variety with a regular locally transitive
$\mathrm{SL}(2)$\,-\,action. Then $X$ is toric if and only if it
is an $\mathrm{SL}(2)/\mathbb{Z}_r$\,-\,embedding with height
$\frac pq$, $(p,q)=1$, such that $r$ is divisible by $q-p.$
\end{theor*}

In Section \ref{cone} the cone of the
$\mathrm{SL}(2)/\mathbb{Z}_{j(q-p)}$\,-\,embedding regarded as a
toric variety is calculated.

After this paper had already been written, D.I.Panyushev informed
the author that the easy part of this theorem claiming that
$\mathrm{SL}(2)$\,-\,embeddings with $r$ divisible by $q-p$ are
toric can be deduced from his results \cite{pan}.

The author is grateful to his supervisor I.V.Arzhantsev for
posing the problem and permanent support.

\

N o t a t i o n

\

$\SL(2)$ is the group of matrices with determinant 1 over the
field $\kk$;

$U=\left\{\begin{pmatrix}1&&0\\ a&&1\end{pmatrix},a\in
\kk\right\}$;

$\mathfrak{X} (G)$ is the group of characters of a group $G$;

$\mathbb{Z}_r$ is the cyclic group of order $r$;

$A^\times$ is the set of invertible elements of an algebra $A$;

$\mathrm{Spec}(A)$ is the spectrum of a finitely generated
algebra $A$;

$\kk[X]$ is the algebra of regular functions on a variety $X$;

$\kk(X)$ is the field of rational functions on a variety $X$;

$(f)$ is a principal divisor of a function $f\in\kk(X)$;

$X/\!\!/G$ is the categorical quotient of a variety $X$ by an
action of a group $G$;

$Cl(X)$ is the divisor class group of a variety $X$.

\section{Preliminary information on toric varieties}
\label{tor} Let $X$ be an (affine) toric variety with a torus $T$
acting on it. By $M$ denote the lattice $\mathfrak{X}(T)$. Let
$N=\mathrm{Hom}\,(M,\mathbb{Z})$ be the dual lattice, with the
pairing between $M$ and $N$ denoted by $\langle\,,\rangle.$ Then N
may be identified with the lattice of one-parameter subgroups of
$T$. Each $m\in M$ gives a character $\chi^m\colon T\rightarrow
\mathbb{K}^*$. The variety $X$ corresponds to a strongly convex
polyhedral cone $\sigma$ in
$N_{\mathbb{Q}}=N\otimes_{\mathbb{Z}}\mathbb{Q}$. Suppose
$\widehat{\sigma}$ is the dual cone in
$M_{\mathbb{Q}}=M\otimes_{\mathbb{Z}}\mathbb{Q}$. Then
$\mathbb{K}[X]$ is isomorphic to the semigroup algebra
$\mathbb{K}[M\cap \widehat{\sigma}]$.

Let $\Delta(1)$ be the set of one-dimensional faces of $\sigma$.
It follows from \cite[Ch.3,~$\S$~3]{ful} that each one-dimensional
face $\rho\in\Delta(1)$ corresponds to a prime $T$-invariant
divisor $D_{\rho}$ in $X$ and each prime $T$-invariant divisor in
$X$ coincides with $D_{\rho}$ for some $\rho\in \Delta(1)$.

The following lemma can be found in \cite[Ch.3,~$\S$~4]{ful}.

\begin{lem}
Any divisor on $X$ is linearly equivalent to a $T$-invariant
divisor.
\end{lem}
\label{Tinv}

Let $n_{\rho}$ be a unique generator of $\rho\cap N$, where
$\rho\in \Delta (1)$. Take $m \in M$. Then $\chi^m\colon T
\rightarrow \mathbb{K}^*$ is a regular function on $T$ and a
rational function on $X$. It is proved in \cite[Ch.3,~$\S$~3]{ful}
that $(\chi^m)=\sum \langle m,n_{\rho}\rangle D_{\rho}$.

The following construction is given in \cite{cox}.

\begin{de}
The Cox ring of a toric variety $X$ is a polynomial ring in $u$
variables enumerated by elements of $\Delta(1)$, where $u$ is the
number of one-dimensional faces in $\sigma$:
$$\mathrm{C}\mathrm{o}\mathrm{x}(X)=\mathbb{K}[x_{\rho} \mid\rho \in \Delta (1) ].$$
\end{de}

A monomial $\prod x_{\rho}^{a_{\rho}}$ corresponds to the
$T$-invariant divisor $D=\sum a_{\rho}D_{\rho}$. Let $x^D$ be this
monomial. Let us grade $\mathrm{C}\mathrm{o}\mathrm{x}(X)$ by the
group $\mathrm{C}\mathrm{l}(X)$ so that the degree of a monomial
$x^D$ is $[D]\in \mathrm{C}\mathrm{l}(X)$. Degrees of two
monomials are equal if these monomials are $\prod
x_{\rho}^{a_{\rho}}$ and $\prod x_{\rho}^{a_{\rho}+ \langle
m,n_{\rho} \rangle}$ for some $m\in M$.

\section {The total coordinate ring}
\label{tot}

Given an arbitrary variety $X$ and a divisor $D$ on it we denote
by $H^0(X,D)$ the following subspace in $\kk(X)$:
$$H^0(X,D)=\{f\in \mathbb{K}(X)| (f)+D\geq 0\}.$$If $X$ is toric
and $D=\sum a_{\rho}D_{\rho}$\, is a $T$\,-\,invariant divisor, we
have
$$H^0(X,D)=\bigoplus_{m\in P_D}\mathbb{K}\chi^m,$$where $P_D=\{m\in M \mid\forall
\rho\in \Delta(1): \langle m,n_{\rho} \rangle \geq - a_{\rho}\}$.

Let $X$ be an arbitrary variety. Suppose that
$\mathrm{C}\mathrm{l}(X)$ is a free finitely generated group. Let
$\{\alpha_1,\ldots, \alpha_s\}$ be its basis and $L_{\alpha_i}$ be
fixed divisors such that $[L_{\alpha_i}]=\alpha_i$. If $\alpha\in
\mathrm{C}\mathrm{l}(X)$, then we have $\alpha=\sum
l_i\alpha_i,\,l_i\in \mathbb{Z}$. We define $L_{\alpha}$ as $\sum
l_i L_{\alpha_i}$. Then we obtain a vector space graded by the
group $\mathrm{C}\mathrm{l}(X)$:
$$S(X)=\bigoplus_{\alpha\in \mathrm{C}\mathrm{l}(X)}H^0(X,L_{\alpha}).$$
Let us define a multiplication in $S(X)$ as follows. If $f\in
H^0(X,L_{\alpha})$, $g\in H^0(X,L_{\beta})$, then their product in
$S(X)$ is an element of $ H^0(X,L_{\alpha+\beta})$ equal to their
product in $\kk(X)$. 
This multiplication is extended to $S(X)$ by distributivity. The
ring $S(X)$ is called {\it the total coordinate ring} of the
variety $X$. It does not depend on basis chosen in
$\mathrm{C}\mathrm{l}(X)$ and on the choice of divisors
$L_{\alpha_i}$ (see \cite{BHH,HK}). If $X$ is a toric variety we
can choose $L_{\alpha_i}$ to be $T$-invariant.

For the following lemma, see \cite{cox}.
\begin{lem}\label {isom}
Assume that $\kk[X]^{\times}=\kk^{\times}$ for a toric variety
$X$. Then the graded algebras $S(X)$ and
$\mathrm{C}\mathrm{o}\mathrm{x}(X)$ are isomorphic.

\end{lem}

The next lemma is a simple corollary of the Rosenlicht theorem
(see \cite{KKV}).

\begin{lem}\label{const}
Assume that a semisimple group G acts locally transitively on a
variety $X$. Then there are no invertible non-constant functions
in $\kk[X]$.
\end{lem}

\begin{corollary}\label{cl}
If for a toric variety $X$ the conditions of Lemma \ref{const} are
fulfilled, then the algebra $S(X)$ is free.
\end{corollary}

Next we consider the problem of lifting of an action of an
algebraic group $G$ on a variety $X$ to an action on the spectrum
of the ring $S(X)$. In this context we can partially invert
Corollary~\ref{cl}.

\begin{prop}\label{G}
Let $G$ be a simply connected semisimple group acting on a variety
$X$ and $H\subset G$ be a closed subgroup. Assume that the
character group $\mathfrak X(H)$ is trivial. Suppose
$G/H\hookrightarrow X$ is an open equivariant embedding. Then
$\mathrm{C}\mathrm{l}(X)$ is free finitely generated and the
following conditions are equivalent:

1) the variety $X$ is toric;

2) there exist a $G$-module $V$ and a linear action of a torus
$\widehat{T}:V$ commuting with the action of $G$ such,
 that
$X\stackrel{G}{\cong}V/\!\!/\widehat{T}$;

3) the algebra $S(X)$ is free.
\end{prop}

\begin{proof}
First we prove that $\mathrm{C}\mathrm{l}(X)$ is free finitely
generated. Denote by $E_1,E_2\ldots E_s$ all prime divisors in
$X\setminus (G/H)$. Since $G$ is connected, all the divisors
$E_i$ are $G$-invariant. Since $G$ is a simply connected
semisimple group and $H$ is its closed subgroup,
$\mathrm{C}\mathrm{l}(G/H)\cong\mathrm{Pic}(G/H)\cong\mathfrak{X}(H)$
(see \cite[Prop.~1]{p2}, \cite[Th.~4]{p2}). Therefore,
$\mathrm{C}\mathrm{l}(G/H) = 0$. Hence $\mathrm{C}\mathrm{l}(X)$
is generated by $G$-invariant divisors, that is, by divisors
$\sum a_iE_i$. If $c_1E_1+\ldots+c_sE_s=(f)$, then, since $E_i$
are $G$-invariant, $f$ is $G$-semiinvariant. But, since $G$ is
semisimple, we have $\mathfrak{X}(G)=0$ and hence $f$ is
$G$-invariant. This implies $f = const$, and
$\mathrm{C}\mathrm{l}(X)$ is freely generated by the divisor
classes $[E_i]$.

Implication {\it 1)$\Rightarrow$ 3)} is already proved in
Corollary~\ref{cl}.

{\it 3)$\Rightarrow$ 2)} The group $G$ acts on $\kk(X)$ as $g\cdot
f(x)=f(g^{-1}\cdot x)$. Consider $D=\sum a_iE_i$ and $f\in
H^0(X,D)$. Then
$$(f)+D\geq 0;  $$
$$(g\cdot f)+D\geq 0\Rightarrow g\cdot
f\in H^0(X,D),$$ that is, $H^0(X,D)$ is a $G$-module. Obviously,
the $G$-action respects the multiplication in $S(X)$.

Let us choose $L_{[E_i]}=E_i$. Then
$$S(X)=\bigoplus_{\alpha\in \mathrm{C}\mathrm{l}(X)}H^0(X,L_{\alpha})
=\bigoplus_{D=a_1E_1+\ldots + a_sE_s}H^0(X,D).$$

Consider an $s$-dimensional torus $\widehat{T}$. Its action on
$H^0(X,D)$ is defined as follows. If $t=(t_1,\ldots,t_s)\in
\widehat{T}$, then $t\cdot f=t_1^{a_1}\ldots t_s^{a_s}f$, while
$f\in H^0(X,D)$, $D=a_1E_1+\ldots +a_sE_s$. This action is
extended to $S(X)$ by linearity. The actions of $G$ and
$\widehat{T}$ on $S(X)$ commute. Therefore,
$$(\widehat{T}\times G): \mathrm{Spec}\,S(X).$$ The
only $G\times\widehat{T}$-invariant functions on $S(X)$ are
constants. Indeed,
$$S(X)^{\widehat{T}\times G}=(S(X)^{\widehat{T}})^G$$
and
$$S(X)^{\widehat{T}}=H^0(X,0)\cong \kk[X].$$
The $G$-action on $X$ has an open orbit. Hence, $\kk[X]^G=\kk$.
Since $S(X)$ is free, we have $\mathrm{Spec}\,S(X)\cong \kk^u$.
By Kraft-Popov's theorem \cite{kp}, an action of a reductive group
on $\mathbb{K}^u$ is equivalent to a linear action, if the
regular functions invariant under this action are only constants.
Now denote the $(\widehat{T}\times G)$-module
$\mathrm{Spec}\,S(X)$ by $V$. Since $S(X)^{\widehat{T}}=\kk[X]$,
$X\stackrel{G}{\cong}V/\!\!/\widehat{T}$.

{\it 2)$\Rightarrow$ 1)} Since the $\widehat{T}$-action on $V$ is
linear, all elements of $\widehat{T}$ are simultaneously
diagonalizable. Consider the torus $\overline{T}$ consisting of
all diagonal operators in the same basis in which $\widehat{T}$ is
diagonal. Then the action
$T=\overline{T}/\widehat{T}:V/\!\!/\widehat{T}$ has an open orbit.
Hence $X$ is a toric variety.

\end{proof}

\section{Lifting of the action}
\label{podnatie}

It follows from the proof of Corollary \ref{G} that, if
$\mathfrak{X}(H)=0$, then  the action $G:X$ may be lifted to an
action on the spectrum of the total coordinate ring. In this
section we define a construction similar to the construction of
total coordinate ring and consider the problem of lifting of the
action, when the group of characters $\mathfrak{X}(H)$ is finite.

Suppose that the divisor class group of a variety $X$ is finitely
generated. Let us choose a generative system
$\xi_1,\ldots,\xi_n,\eta_1,\ldots,\eta_s$ of the group
$$\mathrm{C}\mathrm{l}(X)\cong \mathbb{Z}^n\oplus\mathbb{Z}_{k_1} \oplus\ldots\oplus
\mathbb{Z}_{k_s},$$ where $\{\xi_i\}$ is a basis of the free group
$\mathbb{Z}^n$, and $\eta_j$ is a generator of the cyclic group
$\mathbb{Z}_{k_j}$. Let us fix some divisors
$E_1,\ldots,E_n,W_1,\ldots,W_s$ on $X$, such that $[E_i]=\xi_i$,
$[W_j]=\eta_j.$ Now we can define a generalisation of total
coordinate ring. Consider a linear space
$$S(X)=\bigoplus_{\lambda_i\in \mathbb{Z},\ \mu_j=0,1,\ldots,k_j-1
}H^0(X,\sum \lambda_i E_i+\sum\mu_j W_j).$$ Let us choose rational
functions $F_1,\ldots,F_s$ such that $k_sW_s=(F_s).$ We define a
multiplication $*$ on $S(X)$ as follows. If $f\in H^0(X,\sum a_i
E_i+\sum b_j W_j)$ and $g\in H^0(X,\sum c_i E_i+\sum d_j W_j)$,
then
$$f*g=fg\prod F_j^{[\frac{b_j+d_j}{k_j}]}\in H^0\left(X,\sum
(a_i+c_i) E_i+\sum \left\{\frac{b_j+d_j}{k_j}\right\}k_j
W_j\right) ,$$ where $[x]$ is the integer part of $x$ and $\{x\}$
is the fractional part of $x$. We need to prove that $fg\prod
F_j^{[\frac{b_j+d_j}{k_j}]}$ belongs to $$H^0\left(X,\sum
(a_i+c_i) E_i+\sum \left\{\frac{b_j+d_j}{k_j}\right\}k_j
W_j\right).$$ It follows from
$$
(f)+\sum a_iE_i+\sum b_j W_j \geq 0;
$$
$$
(g)+\sum c_iE_i+\sum d_j W_j \geq 0;
$$
\begin{multline*}
(fg\prod F_j^{[\frac{b_j+d_j}{k_j}]})+\sum(a_i+c_i)E_i+
\sum\left\{\frac{b_j+d_j}{k_j}\right\}k_jW_j =\\=(fg)+\sum
(a_i+c_i)E_i+\sum (b_j+d_j) W_j\geq 0.
\end{multline*}

Extending $*$ by distributivity, we obtain a commutative algebra
$S(X)$. The algebra $S(X)$ is graded by the group
$\mathrm{C}\mathrm{l}(X)$.

\begin{lem}\label{Nezavis}
The algebra $S(X)$ does not depend on the choice of
$E_1,\ldots,E_n,W_1,\ldots,W_s$.

\end{lem}
\begin{proof}
Let
$\widetilde{E}_1,\ldots,\widetilde{E}_n,\widetilde{W}_1,\ldots,\widetilde{W}_s$
be other divisors on $X$ such that $[\widetilde{E}_i]=\xi_i$,
$[\widetilde{W_j}]=\eta_j.$ Then $\widetilde{E}_i=E_i-(J_i)$,
$\widetilde{W}_j=W_j-(R_j).$ There exists a map
$$\varphi\colon H^0(X,\sum \lambda_iE_i+\sum\mu_jW_j )\rightarrow
H^0(X,\sum \lambda_i\widetilde{E}_i+\sum\mu_j\widetilde{W}_j );\
f\mapsto f\prod J_i^{\lambda_i}\prod R_j^{\mu_j}.$$ Let us combine
the maps $\varphi$ with different $\lambda_i$ and $\mu_j$ and
extend this map to a linear map. We obtain a map between the
algebras $S(X)$ corresponding to different choices of divisors
$E_1,\ldots,E_n,W_1,\ldots,W_s$. Suppose $f\in H^0(X,\sum
a_iE_i+\sum b_jW_j )$, $h\in H^0(X,\sum c_iE_i+\sum d_jW_j ).$
Then
\begin{multline*}
\varphi(f)*\varphi(g)=\varphi(f)\varphi(g)\prod
\left(F_jR_j^{-k_j}\right)^{[\frac{b_j+d_j}{k_j}]}=\\=\left(f\prod
J_i^{a_i}\prod R_j^{b_j}\right)\left(h\prod J_i^{c_i}\prod
R_j^{d_j}\right)\prod
\left(F_jR_j^{-k_j}\right)^{[\frac{b_j+d_j}{k_j}]}=\\=
\left(fh\prod F_j^{[\frac{b_j+d_j}{k_j}]}\right)\prod
J_i^{a_i+c_i}\prod R_j^{\{\frac{b_j+d_j}{k_j}\}k_j}=\varphi(f*h).
\end{multline*}
Hence, $\varphi$ is an isomorphism.
\end{proof}
\begin{rem}
The algebra $S(X)$ does not depend on a splitting of
$\mathrm{C}\mathrm{l}(X)$ into a direct sum of cyclic groups and
on the choice of a system of generators of the group
$\mathrm{C}\mathrm{l}(X)$.
\end{rem}

\begin{prop}\label{dejst}
Let $G$ be a semisimple simply connected group acting on a variety
$X$ and $H\subset G$ be a closed subgroup. Assume that the group
of characters $\mathfrak{X}(H)$ is finite. Suppose
$G/H\hookrightarrow X$ is an open equivariant embedding. Then
$\mathrm{C}\mathrm{l}(X)$ is finitely generated and there exists
an action $G:S(X)$ preserving homogeneous components and
coinciding on $S(X)_0\cong \kk[X]$ with the $G$-action on
$\kk[X]$.
\end{prop}
\begin{proof}
We denote by $E_1,\ldots,E_n$ all prime divisors in
$X\setminus(G/H)$. Since $G$ is semisimple, these divisors are
$G$-invariant. From \cite[Prop.~1]{p2} and \cite[Th.~4]{p2} it
follows that $\mathrm{C}\mathrm{l}(G/H)\cong\mathfrak{X}(H)$. The
divisors $E_1,\ldots,E_n$ generate a free subgroup
$\mathbb{Z}^n\subset \mathrm{C}\mathrm{l}(X)$. For any divisor
$D$ in $X$ there exists a linear combination $\sum \lambda_i E_i$
such that the support of $D-\sum \lambda_i E_i$ lies in $G/H.$
Hence, the quotient group $\mathrm{C}\mathrm{l}(X)/\mathbb{Z}^n$
is finite. Therefore, we have
 $$\mathrm{C}\mathrm{l}(X)=\mathbb{Z}^n\oplus\mathbb{Z}_{k_1} \oplus\ldots\oplus
\mathbb{Z}_{k_s}.$$  For any $\alpha\in \mathrm{C}\mathrm{l}(X)$
we have
$\alpha=(\beta_1,\ldots,\beta_n,\alpha_1,\ldots,\alpha_s)$, where
$\beta_i\in \mathbb{Z}$, $\alpha_j\in\mathbb{Z}_{k_j}$. Let us
consider divisors $W_1,\ldots ,W_s$ such that
$[W_j]=(0,\ldots,0,1,0,\ldots,0),$ where 1 stands at the $(n+j)$th
position.

By \cite[Prop.~1]{p2}, $\mathrm{C}\mathrm{l}(G)=0$. This means
that all divisors in $G$ are principal. Suppose $\pi\colon
G\rightarrow G/H$ is the factorisation morphism and $i\colon
G/H\rightarrow X$ is the embedding. These morphisms induce the
embeddings of fields
$\kk(X)\stackrel{i^*}\hookrightarrow\kk(G/H)\stackrel{\pi^*}\hookrightarrow\kk(G).$
The homogeneous space $G/H$ is smooth. Hence any Weil divisor in
$G/H$ is a Cartier divisor. Since $W_j\cap (G/H)$ is a Cartier
divisor, there exists an open covering $\{U_\alpha\}$ of $G/H$
such that $W_j\cap U_\alpha=(\varphi_\alpha).$ Let $Q$ be the
pullback of $W_j\cap (G/H)$ in $G$. We have $Q\cap
\pi^{-1}(U_\alpha)=(\pi^*(\varphi_\alpha)).$ Since $Q$ is a
principal divisor, we have $Q=(f_j),$ where $f_j\in \kk[G].$
Therefore the pullback of the divisor $(F_j|_{G/H})=(i^*(F_j))$
is $(f_j^{k_j})$, that is, $(\pi^*(i^*(F_j)))=(f_j^{k_j}).$ We
may assume that $\pi^*(i^*(F_j))=f_j^{k_j}.$ In the sequel we
shall not distinguish between functions of $\kk(X)$ and their
images under the map $\pi^*\circ i^*.$ Thus, $f_j$ is an element
of $\kk(G)$ such that $f_j^{k_j}\in\kk(X)$ and
$k_jW_j=(f_j^{k_j}).$

Let us define an action $\bullet$ of the group $G$ on $H^0(X,D)$,
where $D=\sum \lambda_iE_i+\sum \mu_jW_j$,
$\lambda_i\in\mathbb{Z}$, $\mu_j=0,\ldots,k_j-1$. Suppose $g\in
G$. Then $g^{-1}\cdot D=\sum \lambda_iE_i+\sum \mu_j(g^{-1}\cdot
W_j).$ Since $k_jW_j=(f_j^{k_j})$, we obtain $k_j(g^{-1}\cdot
W_j)=g^{-1}\cdot (k_jW_j)=g^{-1}\cdot(f_j^{k_j})=(g^{-1}\cdot
f_j^{k_j})=(f_j^{k_j})+\left(\frac{g^{-1}\cdot
f_j^{k_j}}{f_j^{k_j}}\right).$ Let us consider the $H$-action on
$G$ by right shifts. Then $H$ acts on $\kk(G)$ and
$\kk(G/H)=\kk(G)^H.$ Suppose $h\in H$. Since
$f_j^{k_j}\in\kk(G/H)$, we have $h\cdot f_j=\varepsilon f_j,$
where $\varepsilon^{k_j}=1$. Then
$$h\cdot\frac{g^{-1}\cdot
f_j}{f_j}(g')=\frac{f_j(h^{-1}gg')}{f_j(h^{-1}g')}=
\frac{\varepsilon f_j(gg')}{\varepsilon f_j(g')}=\frac{
g^{-1}\cdot f_j}{ f_j}(g').$$ Hence, $\frac{g^{-1}\cdot
f_j}{f_j}\in \kk(G)^H=\kk(X).$ By definition, we put $$g\bullet
f=(g\cdot f) \prod \left(\frac{g\cdot
f_j}{f_j}\right)^{\mu_j}=g\cdot \left(\prod
\left(\frac{f_j}{g^{-1}\cdot f_j}\right)^{\mu_j}f\right),$$ where
$f\in H^0(X,D)$, and the action $\cdot$ is the standard action on
$\mathbb{K}(G)$. Let us check that $g\bullet f$ lies in
$H^0(X,D)$. We have $(f)+D\geq0.$ Hence, $\left(f\prod
\left(\frac{f_j}{g^{-1}\cdot
f_j}\right)^{\mu_j}\right)+g^{-1}\cdot D\geq0.$ Therefore,
$g\cdot\left(f\prod \left(\frac{f_j}{g^{-1}\cdot
f_j}\right)^{\mu_j}\right)+D\geq0$, that is, $g\bullet f\in
H^0(X,D).$ Let us check that $\bullet$ is an action.

\begin{multline*}
(g\widetilde{g})\bullet
f=(g\widetilde{g})\left(f\prod\left(\frac{f_j}
{(g\widetilde{g})^{-1}\cdot f_j}\right)^{\mu_j}\right)=
\\
=g\cdot\left(\widetilde{g}\cdot\left(f\prod\left(
\frac{f_j}{\widetilde{g}^{-1}\cdot f_j} \right)^{\mu_j}
\prod\left(\frac{\widetilde{g}^{-1}\cdot
f_j}{\widetilde{g}^{-1}\cdot(g^{-1}\cdot f_j)}
\right)^{\mu_j}\right)\right)=\\
=g\cdot\left((\widetilde{g}\bullet f )\left(\frac {f_j}
{g^{-1}\cdot
f_j}\right)^{\mu_j}\right)=g\bullet(\widetilde{g}\bullet f).
\end{multline*}
For $(\mu_1,\ldots,\mu_s)=(0,\ldots,0)$ the action $\bullet$
coincides with the action $\cdot$ on $\kk[X]$. Let us extend the
action $\bullet$ to the linear action on $S(X)$.

Let us prove that $\bullet$ is an action by automorphisms of
algebra $S(X)$. We need to show that $g\bullet(f*f')=(g\bullet
f)*(g\bullet f')$ for any $g\in G$ and for any $f$ and $f'$ in
$S(X)$. Since the action $\bullet$ is linear we may check only
$g\bullet(\chi^m*\chi^n)=(g\bullet \chi^m)*(g\bullet \chi^n)$,
where $\chi^m\in H^0(X,\sum a_iE_i+\sum b_jW_j)$, $\chi^n\in
H^0(X,\sum c_iE_i+\sum d_jW_j)$. By definition,
$$\chi^m*\chi^n=\chi^m\chi^n\prod F_j^{[\frac{b_j+d_j}{k_j}]}\in
H^0\left(X,\sum(a_i+c_i)E_i+\left\{\frac{b_j+d_j}{k_j}\right\}k_jW_j\right).$$
We have
\begin{multline*}
g\bullet(\chi^m*\chi^n)=\left(g\cdot\left(\chi^m\chi^n\prod
F_j^{[\frac{b_j+d_j}{k_j}]}\right)\right)\prod\left(\frac{g\cdot
f_i}{f_i}\right)^{\{\frac{b_j+d_j}{k_j}\}k_j}=\\
=\frac{(g\cdot (\chi^m \chi^n))\prod(g\cdot f_j)^{b_j+d_j}}
{\prod f_j^{\{\frac{b_j+d_j}{k_j}\}k_j}}=\\
=\left((g\cdot\chi^m)\prod\left(\frac{g\cdot
f_i}{f_i}\right)^{b_j}\right)\left((g\cdot\chi^n)\prod\left(\frac{g\cdot
f_i}{f_i}\right)^{d_j}\right)\prod
F_j^{[\frac{b_j+d_j}{k_j}]}=(g\bullet f)*(g\bullet f').
\end{multline*}

Proposition \ref{dejst} is proved.
\end{proof}

If the algebra $S(X)$ is finitely generated we can consider
$\mathrm{Spec}\,S(X)$.

\begin{lem}\label{ratz}
The action $\bullet$ on $S(X)$ induces a regular action on
$\mathrm{Spec}\,S(X)$.
\end{lem}
\begin{proof}
We need to prove that $S(X)$ is a rational $G$-module. Let us
prove that $H^0(X,D)$ is a rational $G$-module with respect to
the action $\bullet$. Here $D=\sum a_iE_i+\sum b_jW_j$ and the
divisors $E_i$ are $G$-invariant. Let us recall that there are
embeddings of fields
$\kk(X)\stackrel{i^*}\hookrightarrow\kk(G/H)\stackrel{\pi^*}\hookrightarrow\kk(G).$
If $f\in H^0(X,D)$, then $i^*(f)\in H^0(G/H,D\cap(G/H)).$ The
pullback of the divisor $D\cap(G/H)$ under the map $\pi$ is a
principal divisor $(J)$ in $G$. Hence we obtain $f\in
H^0(G,(J)),$ that is, $fJ\in \kk[G].$ (Recall that we do not
distinguish between $f\in\kk(X)$ and $\pi^*(i^*(f))\in\kk(G)$.)
The action $\bullet$ on $H^0(X,D)$ is defined by the formula
$$g\bullet f=g\cdot\left(f\prod\left(\frac{
f_j}{g^{-1}f_j}\right)^{b_j}\right)=\frac{g\cdot(f\prod
f_j^{b_j})}{\prod f_j^{b_j}}.$$ By definition,
$\pi^*(W_j\cap(G/H))=(f_j).$ Hence, $\pi^*(D\cap(G/H))=(\prod
f_j^{b_j}),$ that is, we may assume $J=\prod f_j^{b_j}.$ Then
$g\bullet f=\frac{g\cdot(fJ)}{J}.$ But it has been proved above
that $fJ\in \kk[G].$ It is known that $\kk[G]$ is a rational
module. Therefore, $H^0(X,D)$ is a rational $G$-module with
respect to the action~$\bullet$. Hence, $S(X)$ is a rational
$G$-module.
\end{proof}

Let $X$ be a toric variety. Then Lemmas \ref{Tinv} and
\ref{Nezavis} imply that the divisors
$E_1,\ldots,E_n,W_1,\ldots,W_s$ in the definition of $S(X)$ can
be chosen to be $T$-invariant.
\begin{lem}\label{S_Cox}
The graded algebras $S(X)$ and $\mathrm{C}\mathrm{o}\mathrm{x}(X)$
are isomorphic.
\end{lem}
\begin{proof}
Consider a divisor $D=\sum a_iE_i+\sum b_j W_j$. Recall that
$$H^0(X,D)=\bigoplus_{m\in P_D} \mathbb{K}\chi^m.$$ Let $\psi_D$
be the linear map from $H^0(X,D)$ to
$\mathrm{C}\mathrm{o}\mathrm{x}(X)_{[D]}$ taking each $\chi^m$,
where $m\in P_D$, to $x^{D+(\chi^m)}$. Since the only invertible
functions in $\kk[X]$ are constants, different functions $\chi^m$
corresponds to different divisors $(\chi^m)$. Therefore, $\psi_D$
is an embedding. If $x^{D_1}\in
\mathrm{C}\mathrm{o}\mathrm{x}(X)_{[D]}$, then $D_1-D$ is a
principal divisor. Suppose $D=\sum \alpha_\rho D_\rho$, $D_1=\sum
\beta_\rho D_\rho$. Then $\beta_\rho-\alpha_\rho=\langle m_0,
n_\rho\rangle$ for some $m_0\in M$. But $\prod
x_\rho^{\beta_{\rho}}$ is a monomial in
$\mathrm{C}\mathrm{o}\mathrm{x}(X)$. Hence for each
$\rho\in\Delta(1)$ we have $\beta_{\rho}\geq0$. Therefore
$\langle m_0,n_\rho\rangle \geq -\alpha_\rho$, that is, $m_0$
lies in $P_D$. Consequently $\psi_D$ is a surjection. Thus
$\psi_D$ is an isomorphism of linear spaces. Let us combine all
maps $\psi_{(\sum a_iE_i+\sum b_jW_j)}, a_i\in \mathbb{Z},
b_j=0,\ldots,k_j-1$. We obtain the isomorphism of linear spaces
$\psi\colon S(X)\rightarrow \mathrm{C}\mathrm{o}\mathrm{x}(X).$
Let us prove that $\psi$ is an isomorphism of algebras. We need
to check that $\psi(f*g)=\psi(f)\psi(g)$ for any $f$ and $g$ in
$S(X)$. By distributivity we may assume that $f=\chi^m\in
H^0(X,\sum a_iE_i+\sum b_jW_j)$ and $g=\chi^n\in H^0(X,\sum
c_iE_i+\sum d_jW_j)$. Then by the definitions of the
multiplication $*$ and of the isomorphism $\psi$ we obtain
\begin{multline*}
\psi(f*g)=\psi(\chi^m*\chi^n)=\psi\left(\chi^{m+n}\prod F_j^
{[\frac{b_j+d_j}{k_j}]}\right)=\\
=x^{\sum(a_i+c_i)E_i+\sum\{\frac{b_j+d_j}{k_j}\}k_jW_j+(\chi^m)+
(\chi^n)+\sum[\frac{b_j+d_j}{k_j}](F_j)}=\\
=x^{\sum(a_i+c_i)E_i+\sum\{\frac{b_j+d_j}{k_j}\}k_jW_j+(\chi^m)+
(\chi^n)+\sum[\frac{b_j+d_j}{k_j}]k_jW_j}=\\
=x^{\sum(a_i+c_i)E_i+\sum(b_j+d_j)W_j+(\chi^m)+(\chi^n)}=\\
=x^{\sum a_iE_i+\sum b_jW_j+(\chi^m)}x^{\sum c_iE_i+\sum
d_jW_j+(\chi^n)}= \psi(\chi^m)\psi(\chi^n)=\psi(f)\psi(g).
\end{multline*}
\end{proof}

\section {A necessary and sufficient condition for an affine $G/H$-embedding to be toric.}
\label{krit}

\begin{theor}\label{GH}Let $G$ be a semisimple simply connected group acting on a variety
$X$. Suppose $H$ is a closed subgroup in $G$ such that the group
of characters $\mathfrak{X}(H)$ is finite. If there is an open
equivariant embedding $G/H\hookrightarrow X$, then the following
conditions are equivalent:

1) X is a toric variety;

2) There exist a $G$-module $V$ and a linear action of a
quasitorus $\widehat{T}$ on $V$ commuting with the $G$-action such
that $X\stackrel{G}{\cong}V/\!\!/\widehat{T}$;

3) The algebra $S(X)$ is free.

\end{theor}

\begin{proof}
The implication {\it 1) $\Rightarrow$ 3)} follows from Lemma
\ref{S_Cox}.

{\it 3)$\Rightarrow$ 2)} The algebra $S(X)$ is graded by the
group $\mathrm{C}\mathrm{l}(X)\cong
\mathbb{Z}^n\oplus\mathbb{Z}_{k_1}\oplus\ldots\oplus\mathbb{Z}_{k_s}$.
Proposition \ref{dejst} implies that there exists a $G$-action by
automorphisms of the ring $S(X)$ that preserves homogeneous
components and coincides on $S(X)_{(0,\ldots,0)}\cong
\mathbb{K}[X]$ with the standard $G$-action. Consider a
quasitorus $\widehat{T}=\widetilde{T}\times
\mathbb{Z}_{k_1}\times\ldots\times\mathbb{Z}_{k_s}$, where
$\widetilde{T}$ is an $n$-dimensional torus. Let us define a
$\widehat{T}$-action on $S(X)$. Suppose $f\in
S(X)_{(\lambda_1,\ldots,\lambda_n,\mu_1,\ldots,\mu_s)}$, where
$\lambda_i\in \mathbb{Z}$, $\mu_j=0,1,\ldots,k_j-1$. Let
$t=(t_1,\ldots,t_n)\in \widetilde{T}$, $\varepsilon_j\in
\mathbb{Z}_{k_j}\subset \mathbb{K}^*$. Then
$(t,\varepsilon_1,\ldots,\varepsilon_s)\cdot f=\prod
t_i^{\lambda_i}\prod \varepsilon_j^{\mu_j} f.$ Obviously, this is
an action. Moreover, this action commute with $G$-action, because
$G$ preserves homogeneous components in $S(X)$ and every element
of $\widehat{T}$ acts on each component by multiplying by a
constant. Hence $G\times\widehat{T}$ acts on $S(X)$. Lemma
\ref{ratz} implies that $G\times\widehat{T}$ acts on
$V=\mathrm{Spec}\,S(X)$. Since $S(X)$ is a free algebra, $V$ is a
linear space. We have $S(X)^{G\times\widehat{T}}
=(S(X)^{\widehat{T}}) ^{G}
=S(X)_{(0,\ldots,0)}^{G}=\mathbb{K}[X]^{G}=\mathbb{K}$. The last
equality holds because $G$ acts on $X$ with an open orbit. The
Kraft-Popov theorem \cite{kp} implies that the action
$G\times\widehat{T}\colon V$ is equivalent to an linear action.
Note that $S(X)^{\widehat{T}}= S(X)_{(0,\ldots,0)}$. Moreover,
since $G$-action on $S(X)_{(0,\ldots,0)}\cong \mathbb{K}[X]$
coincides with the standard action, we have
$S(X)^{\widehat{T}}=S(X)_{(0,\ldots,0)}\stackrel{G}\cong
\mathbb{K}[X]$. Hence $V/\!\!/\widehat{T}\stackrel{G}\cong X$.

{\it 2) $\Rightarrow$ 1)} The proof coincides with the proof of
the corresponding implication in Proposition \ref{G}.
\end{proof}

\begin{rem}\label{zam}
It follows from the proofs of Proposition \ref{dejst} and Theorem
\ref{GH} that $n$ is the number of prime divisors in the
complement to $G/H$ and $s$ is not larger then the number of
cyclic components in any splitting of $\mathfrak X(H)$.
\end{rem}

\section{Necessary information on
$\mathrm{SL}(2)/\mathbb{Z}_r$\,-\,embeddings} The results
mentioned in this section can be found in \cite{p1} (see also
\cite[chapter 3]{kr}).

A normal affine $\mathrm{SL}(2)$\,-\,embedding $X$ is unique up
to an isomorphism determined by its height. The height is a
rational number defined as follows. Let us consider an
$\mathrm{SL}(2)$\,-\,equivariant open embedding $\varphi\colon
\mathrm{SL}(2)\hookrightarrow X$. It induces an embedding
$\varphi^*\colon\mathbb{K}[X]\hookrightarrow
\mathbb{K}[\mathrm{SL}(2)]
=\mathbb{K}[\alpha,\beta,\gamma,\delta]/(\alpha\delta-\beta\gamma-1).$
Here $\alpha,\beta,\gamma$ and $\delta$ are the functions such
that $\alpha(A)=a_{11},\,\beta(A)=a_{12},\,\gamma(A)=a_{21}
,\,\delta(A)=a_{22}$ for any
$A=\begin{pmatrix}a_{11}&&a_{12}\\a_{21}&&a_{22}\end{pmatrix}\in
SL(2)$. Let us consider the unipotent subgroup $U=\left\{\begin{pmatrix} 1&0\\
*&1\end{pmatrix}\right\}$ in $\mathrm{SL}(2)$. U acts on
$\mathrm{SL}(2)$ by left shifts. The algebra of $U$-invariant
functions on $\mathrm{SL}(2)$ is
$\mathbb{K}[\SL(2)]^U=\mathbb{K}[\alpha,\beta]$. Let us consider
the restriction of $\varphi^*$ to $\mathbb{K}[X]^U$. We have
$\varphi^*:\mathbb{K}[X]^U\hookrightarrow \mathbb{K}[\SL(2)]^U
=\mathbb{K}[\alpha,\beta]$.

\begin{prop} The image $\varphi^*(\kk[X]^U)$ is a monomial
subalgebra in $\kk[\alpha,\beta]$. Moreover,
$\varphi^*(\kk[X]^U)=\langle\alpha^i\beta^j \mid j/i\leq
h\rangle$ for some rational number $h$.
\end{prop}

The rational number $h$ is called the {\it height} of the
embedding $X$.

For any $\SL(2)/\mathbb{Z}_r$-embedding $X$ there exists a unique
up to an isomorphism $\SL(2)$-embedding $Y$ such that $X\cong
Y/\!\!/\mathbb{Z}_r$. Here the $\mathbb{Z}_r$-action on $Y$ is an
extension of the $\mathbb{Z}_r\subset \SL(2)$ on $\SL(2)$ by
right shifts. The height of the corresponding $\SL(2)$-embedding
$Y$ is called the {\it height} of the
$\SL(2)/\mathbb{Z}_r$-embedding $X$. The order of the stabilizer
of a point in the open orbit is called the {\it degree} of $X$.
Thus, any normal affine $\SL(2)/\mathbb{Z}_r$-embedding is unique
up to isomorphism determined by its height
$h\in\mathbb{Q}\cap(0,1]$ and degree $r\in\mathbb{N}$.

\begin{prop}
There is a unique prime divisor in the complement of the open
orbit in a normal affine $\SL(2)/\mathbb{Z}_r$-embedding.
\end{prop}
The divisor class group of a normal affine
$\SL(2)/\mathbb{Z}_r$-embedding has been calculated in
\cite{pan}. It is isomorphic to $\mathbb{Z}\oplus\mathbb{Z}_l,$
where $l=\frac{r}{(r,q-p)}.$

\section {Classification of three-dimensional
toric varieties admitting a locally transitive $\SL(2)$-action}
\label{class}
\begin{prop}
Let $G$ be a semisimple simply connected  group and $H$ be a
closed subgroup of $G$. Assume that $\mathfrak{X}(H)$ is finite.
If the homogeneous space $G/H$ has positive dimension, then $G/H$
is not a toric variety.
\end{prop}

\begin{proof}
Let $X=G/H$ be a toric variety of positive dimension. It follows
from Theorem \ref{GH} that $X\stackrel{G}\cong V/\!\!/L,$ where
$V$ is a vector space and $L$ is a group acting linearly on $V$.
Hence the image of $0\in V$ under the factorisation morphism is
an $G$-stable point. This contradicts the transitivity of the
$G$-action on $X$.
\end{proof}
\begin{corollary}\label{new}
No three-dimensional toric variety admits a transitive
$\mathrm{SL}(2)$-action.
\end{corollary}

Now suppose that the $\mathrm{SL}(2)$-action on a
three-dimensional variety $X$ is locally transitive but is not
transitive. Then $X$ is an
$\mathrm{SL}(2)/\mathbb{Z}_r$-embedding. Every
$\SL(2)/\mathbb{Z}_r$-embedding is uniquely determined by its
height $h$. Let us calculate the heights of the toric
$\SL(2)/\mathbb{Z}_r$-embeddings. Let us consider a toric
$\SL(2)/\mathbb{Z}_r$-embedding $X$. A toric variety $X$
corresponds to a cone $\sigma$ in the space $N_{\mathbb{Q}}.$

\begin{lem}
The number $u$ of one-dimensional edges in $\sigma$ is equal to
$4$.

\end{lem}
\begin{proof}
By Lemma \ref{Tinv}, any divisor in $X$ is equivalent to a
$T$-invariant one. Any $T$-invariant divisor in $X$ can be
written as $\sum a_\rho D_\rho.$ Hence, the rank of the group of
$T$-invariant Weil divisors in $X$ is $u$. A $T$-invariant
divisor is principal if and only if it can be written as
$\sum\langle m,n_\rho\rangle D_\rho$ for some $m\in M$. The
lattice $M$ has dimension three. Hence the rank of the group of
principal $T$-invariant divisors is three. Therefore the rank of
the divisor class group is $u-3=1$, that is, $u=4.$
\end{proof}

Recall that there exists a vector space $V\cong
\mathbb{K}^u=\kk^4$ and a linear $(\mathrm{SL}_2\times
\widehat{T})$-action on $V$ such that $X\cong V/\!\!/\widehat{T}$.
By Remark \ref{zam}, we have $\widehat{T}=\widetilde{T}\times
\mathbb{Z}_l$, where $\widetilde{T}$ is a one-dimensional torus.
Suppose $V=V_1\oplus\ldots\oplus V_n$, where $V_i$ are the weight
spaces of $\widetilde{T}$-action. Then each $V_i$ is an
$\mathrm{SL}(2)$-module.

Let us consider the space $R_s=\langle x^s,x^{s-1}y,\ldots
,y^s\rangle$ of binary forms of degree $s$. The group
$\mathrm{SL}(2)$ acts on $R_s$ by the following rule. If
$g=\begin{pmatrix} a&b\\c&d\end{pmatrix}$, then $g\cdot x=dx-by$;
$g\cdot y=-cx+ay$. Then $R_s$ is an irreducible
$\mathrm{SL}_2$-module of dimension $s+1$. Any irredusible
$(s+1)$-dimensional $\mathrm{SL}_2$-module is isomorphic to $R_s$.

There are 5 cases of splitting of $V$ into irreducible
$\mathrm{SL}_2$-modules.

\noindent1) $V=R_0\oplus R_0\oplus R_0\oplus R_0$, \ 2) $V=
R_0\oplus R_0\oplus R_1$, \ 3) $V= R_0\oplus R_2$, \ 4) $V=
R_1\oplus R_1$, 5) $V=R_3.$

In cases 1), 2) and 3) all orbits in $V$ have dimension less then
three. Hence, all orbits in $V/\!\!/\widehat{T}$ have dimension
less then three.

In case 5) $\widetilde{T}$ acts on $V$ by homotatis. Hence
$V/\!\!/\widetilde{T}$ is a point. Thus case 5) is impossible.

Let us consider case $4)\colon V= R_1\oplus R_1$. The variety
$V/\!\!/\widetilde{T}$ can be three-dimensional only if the
weights of the $\widetilde{T}$-action have opposite signs.

We have

$$V=\mathbb{K}^2\oplus \mathbb{K}^2=V_1\oplus V_2;$$ $$
\widetilde{T} :V_1,\qquad t\cdot v_1=t^{np}v_1,\ p\in\mathbb{N},\
v_1\in V_1;
$$$$
\widetilde{T} :V_2,\qquad t\cdot v_2=t^{-nq}v_2,\ q\in\mathbb{N},\
v_2\in V_2,\  n\in \mathbb{N},(p,q)=1.
$$
Let us introduce the following notation.
$$ Z=V/\!\!/\widetilde{T};$$$$Y=\{\underbrace{v_1\otimes\ldots\otimes v_1}_q\otimes
\underbrace{v_2\otimes\ldots\otimes v_2}_p\}\subset V_1^{\otimes
q}\otimes V_2^{\otimes p}.$$
\begin{prop}There exists an
$\mathrm{SL}_2$-equivariant isomorphism between the varieties $Y$
and $Z$.
\end{prop}
\begin{proof}
Let us fix the isomorphism of $\mathrm{SL}_2$-modules $V_1\cong
V_2$. Let us also fix the bases in $V_1$ and $V_2$ corresponding
to each other under the chosen isomorphism. Let $x_1,x_2$ be the
coordinates in $V_1$ corresponding to the chosen basis and
$y_1,y_2$ be the corresponding coordinates in $V_2$. Then
$$\mathbb{K}[Y]=\kk[Z]=\mathbb{K}[x_1^mx_2^{q-m}y_1^ly_2^{p-l}\mid
m=0,\ldots,q;\ l=0,\ldots,p].$$ Hence, $Y\stackrel{\SL(2)}\cong
Z$.

\end{proof}
In the sequel we shall not distinguish between $Z$ and $Y$ and
use the notation $Y$.

Let us consider the following $\mathrm{SL}_2$-orbit in $Y$
$$ \mathcal{O}=Orb(u),\ u= \underbrace{e_1\otimes\ldots\otimes
e_1}_q\otimes\underbrace{ e_2\otimes\ldots\otimes e_2}_p.$$ Here
$e_1\in V_1$ and $e_2\in V_2$ are not proportional after the
identification of $V_1$ and $V_2$ by the chosen isomorphism. The
stabilizer of $u$ is $St_u=\left\{\begin{pmatrix}
\varepsilon&0\\0&\varepsilon^{-1}\end{pmatrix},\varepsilon^{|q-p|}=1\right\}$.
If $p\neq q$, then $St_u\cong \mathbb{Z}_{q-p}$ is a finite
group. Hence $Orb(u)$ is three-dimensional and open. If $p=q$,
then the stabilizer of any point contains a one-dimensional torus.
Therefore there is no open orbit. In the sequel we shall consider
the case $q>p$. Then the toric variety $Y=V/\!\!/\widetilde{T}$
contains an open $\SL(2)$-orbit. Therefore, $Y$ is a toric
$\SL(2)/\mathbb{Z}_{q-p}$\,-\,embedding.

The group $\SL(2)$ acts on itself by left shifts. This action
induces the action on the homogeneous space $SL(2)/\mathbb{Z}_r$.
Let us consider the $\SL(2)$-equivariant dominant morphism
$$\varphi\colon \SL(2)\rightarrow Y;\ g \stackrel{\varphi}\mapsto
g\cdot  u.$$ It corresponds to the $\SL(2)$-equivariant embedding
$$\varphi^*:\mathbb{K}[Y]\hookrightarrow
\mathbb{K}[\SL(2)]=\mathbb{K}[\alpha,\beta,\gamma,\delta]/
(\alpha\delta-\beta\gamma-1).$$ Suppose $g=\begin{pmatrix}
a&b\\c&d \end{pmatrix} \in \SL(2)$. Then $g\cdot e_1=ae_1+ce_2,\
g\cdot e_2=be_1+de_2.$ Note that
$g\cdot\alpha=d\alpha-b\gamma,g\cdot\gamma=
a\gamma-c\alpha,g\cdot\beta=d\beta-b\delta,g\cdot\delta=
a\delta-c\beta.$ We have
\begin{multline*}
\varphi^*(x_1^mx_2^{q-m}y_1^ly_2^{p-l})(g)=
x_1^mx_2^{q-m}y_1^ly_2^{p-l}(\varphi(g))=\\
=x_1^mx_2^{q-m}y_1^ly_2^{p-l} (\underbrace{g\cdot
e_1\otimes\ldots\otimes g\cdot e_1}_q\otimes \underbrace{ g\cdot
e_2\otimes\ldots\otimes g\cdot e_2}_p)=\\=
x_1^mx_2^{q-m}y_1^ly_2^{p-l}
(\underbrace{(ae_1+ce_2)\otimes\ldots\otimes (ae_1+ce_2)}_q\otimes
\underbrace{ (be_1+de_2)\otimes\ldots\otimes (be_1+de_2)}_p)=\\=
a^mc^{q-m}b^ld^{p-l}=\alpha^m\gamma^{q-m}\beta^l\delta^{p-l}(g).
\end{multline*}
Thus $\varphi^*(x_1^mx_2^{q-m}y_1^ly_2^{p-l})=
\alpha^m\gamma^{q-m}\beta^l\delta^{p-l}$.

Let us
consider a unipotent subgroup $U=\left\{\begin{pmatrix} 1&0\\
*&1\end{pmatrix}\right\}$ in $\SL(2)$. The algebra of
$U$-invariant functions on $\mathrm{SL}(2)$ is
$\mathbb{K}[\SL(2)]^U=\mathbb{K}[\alpha,\beta]$. The restriction
of $\varphi^*$ to $\mathbb{K}[Y]^U$ is
$\varphi^*|_{\mathbb{K}[Y]^U}\colon\mathbb{K}[Y]^U\hookrightarrow
\mathbb{K}[\SL(2)]^U =\mathbb{K}[\alpha,\beta]$. Let us prove
that the height of $Y$ is equal to the maximal possible value of
the fraction $w/t$, where $\alpha^t\beta^w\in
\mathrm{Im}\,\varphi^*|_{\mathbb{K}[Y]^U}=\mathrm{Im}\,\varphi^*\cap
\mathbb{K}[\alpha,\beta]$. Let $W$ be the toric
$\SL(2)$-embedding such that $Y=W/\!\!/\mathbb{Z}_s$. Then we
have dominant morphisms
$$\SL(2)\stackrel{\psi}\rightarrow W\stackrel{\pi}\rightarrow
Y,$$where $\varphi=\pi\circ\psi$. Hence,
$$\mathbb{K}[Y]^U
\stackrel{\pi^*|_{\mathbb{K}[W]^U}}\hookrightarrow
\mathbb{K}[W]^U\stackrel{\psi^*|_{\mathbb{K}[\SL(2)]^U}}\hookrightarrow
\mathbb{K}[\alpha,\beta] .$$ If $\alpha^t\beta^w\in
\mathrm{Im}\,(\varphi^*)\cap \mathbb{K}[\alpha,\beta]$, then
$\alpha^t\beta^w\in \mathrm{Im}\,(\psi^*)\cap
\mathbb{K}[\alpha,\beta]$. Therefore, the height of $W$, which is
equal to the height of $Y$, is not less then $w/t$. If the height
of $W$ is equal to $h$, then there exists a monomial
$\alpha^\xi\beta^\eta$ in $\mathrm{Im}\,(\psi^*)\cap
\mathbb{K}[\alpha,\beta]$ such that $\eta/\xi=h$. Then
$\alpha^\xi\beta^\eta=\psi^*(f)$ for some $f\in \mathbb{K}[W]^U$.
Let
$\zeta=\begin{pmatrix}\varepsilon&&0\\0&&\varepsilon^{-1}\end{pmatrix}$
be a generator of $\mathbb{Z}_s$. Suppose
$$\widetilde{f}=f(\zeta\cdot f) (\zeta^2\cdot
f)\ldots(\zeta^{s-1}\cdot f)\in \mathbb{K}[Y]^U.$$ It is easy to
see that $\varphi^*(\widetilde{f})=\alpha^{s\xi}\beta^{s\eta}$.
But $\frac{s\eta}{s\xi}=\frac\eta\xi=h.$ Thus, the height of $Y$
is the maximal possible value of $w/t$.
\begin{prop}
If a monomial $\alpha^t\beta^w$ belongs to the image of
$\varphi^*$, then $w/t\leq p/q$.
\end{prop}
\begin{proof}
Assume that $\varphi^*(\sum
z_i{x_1}^{m_i}{x_2}^{q-m_i}{y_1}^{l_i}{y_2}^{p-l_i})=\alpha^t\beta^w,
z_i\in \mathbb{K}.$ Then $$\sum
z_i\alpha^{m_i}\gamma^{q-m_i}\beta^{l_i}\delta^{p-l_i}=
\alpha^t\beta^w+(\alpha\delta-\beta\gamma-1)
F(\alpha,\beta,\gamma,\delta).$$ (This equality holds in
$\mathbb{K}[\alpha,\beta,\gamma,\delta]$.) Putting $\gamma=0$ we
obtain
$$\sum
z_i\alpha^{q}\beta^{l_i}\delta^{p-l_i}=\alpha^t\beta^w+(\alpha\delta-1)
\widetilde{F}(\alpha,\beta,\gamma).$$
Substituting $1/\alpha$ for $\delta$ we obtain
$$\sum z_i\alpha^{q-p+l_i}\beta^{l_i}=\alpha^t\beta^w.$$
Therefore,
$$ t=q-p+l_i;\ w=l_i;\
w/t=l_i/(q-p+l_i)\leq p/q .$$
\end{proof}
For the monomial $\alpha^q\beta^p=\varphi^*(x_1^qy_1^p)$, we have
$w/t=p/q$. Hence the height of $Y$ is equal to $p/q$. But, by
definition, the height of $Y$ is equal to the height of $W$. And
it is equal to the height of $X$. Therefore the height of $X$ is
equal to $p/q$. Recall that $Y=V/\!\!/\widetilde{T}$ is an
$\SL(2)/\mathbb{Z}_{q-p}$-embedding, and
$X=(V/\!\!/\widetilde{T})/\!\!/\mathbb{Z}_l=Y/\!\!/\mathbb{Z}_l$.
Hence the order of the stabilizer of a point in the open
$\SL(2)$-orbit in $X$ is divisible by $q-p$.

Let us check that an $\SL(2)/\mathbb{Z}_{(q-p)l}$-embedding of
height $p/q$ is toric for every positive integer $l$. Indeed,
$Y=\mathbb{K}^4/\!\!/\widetilde{T}$ is a toric variety with the
height $p/q$, that is, $Y$ corresponds to $l=1$. Let us consider
an $\SL(2)/\mathbb{Z}_{(q-p)l}$-embedding with height $p/q$. We
have
$X=W/\!\!/\mathbb{Z}_{(q-p)l}=(W/\!\!/\mathbb{Z}_{(q-p)})/\!\!/
\mathbb{Z}_{l}=Y/\!\!/\mathbb{Z}_l.$ Suppose $\pi\colon
Y\rightarrow X$ is the factorisation morphism. The torus
$T=\overline{T}/\widetilde{T}$ acts on $Y$ with an open orbit and
trivial stabilizer of generic point. Let us define a $T$-action
on $X$ by $t\cdot \pi(y)=\pi(t\cdot y)$. This action is well
defined if the actions of $T$ and $\mathbb{Z}_l$ on $Y$ commute.
Let us check that the actions of $T$ and $\mathbb{Z}_l=
\mathbb{Z}_{(q-p)l}/\mathbb{Z}_{q-p}$ on $\mathbb{K}[Y]$ commute,
where $\mathbb{Z}_{(q-p)}\subset \SL(2)$ and
$\mathbb{Z}_{(q-p)l}\subset \SL(2)$. Suppose
$$
t=\overline{t}\widetilde{T}\in T,\
\overline{t}=\begin{pmatrix}a&&0&&0&&0\\0&&b&&0&&0\\0&&0&&c&&0\\0&&0&&0&&d\end{pmatrix}
\in \overline{T},
$$
$$
\overline{\zeta}=\begin{pmatrix}\varepsilon&&0\\0&&\varepsilon^{-1}
\end{pmatrix}\in\mathbb{Z}_{l(q-p)},\ \zeta=\overline{\zeta}
\mathbb{Z}_{q-p}\in \mathbb{Z}_l.
$$
Here $\varepsilon$ is a $(q-p)l$-th root of unity. Then

$$t\cdot
x_1^mx_2^{q-m}y_1^ny_2^{p-n}=a^{-m}b^{m-q}c^{-n}d^{n-p}
x_1^mx_2^{q-m}y_1^ny_2^{p-n}.$$ Recall that
$$Y=\{v_1\otimes\ldots\otimes v_1\otimes v_2\otimes\ldots\otimes
v_2\}\subset V_1^{\otimes q}\otimes V_2^{\otimes p}.$$ The open
$\SL(2)$-orbit is $\{g\cdot e_1\otimes\ldots\otimes g\cdot
e_1\otimes g\cdot e_2\otimes\ldots\otimes g\cdot e_2\}$, where
$(e_1,e_2)$ is a fixed basis in $\mathbb{K}^2$. Suppose
$v=g_0\cdot e_1\otimes\ldots\otimes g_0\cdot e_1\otimes g_0\cdot
e_2\otimes\ldots\otimes g_0\cdot e_2.$ Then
\begin{multline*}\!\!\!\!\!\!\!\zeta\cdot v=g_0\overline{\zeta}
\cdot e_1\otimes\ldots\otimes g_0\overline{\zeta}\cdot e_1\otimes
g_0\overline{\zeta}\cdot e_2 \otimes\ldots\otimes
g_0\overline{\zeta}\cdot e_2=\\=\varepsilon g_0\cdot
e_1\otimes\ldots\otimes\varepsilon g_0\cdot
e_1\otimes\varepsilon^{-1} g_0\cdot
e_2\otimes\ldots\otimes\varepsilon^{-1} g_0\cdot
e_2=\\=\varepsilon^{q-p}g_0\cdot e_1\otimes\ldots\otimes g_0\cdot
e_1\otimes g_0\cdot e_2\otimes\ldots\otimes g_0\cdot
e_2=\varepsilon^{q-p}v.\end{multline*} Hence, $\zeta\cdot
x_1^mx_2^{q-m}y_1^ny_2^{p-n}(v)=x_1^mx_2^{q-m}y_1^ny_2^{p-n}(\zeta^{-1}
\cdot v)=\varepsilon^{p-q}x_1^mx_2^{q-m}y_1^ny_2^{p-n}(v).$ This
implies that the actions of the torus and $\mathbb{Z}_l$ commute.
Thus, $X$ is a toric $\SL(2)/\mathbb{Z}_{(q-p)l}$-embedding.

\noindent Let us formulate the main result.
\begin{theor}
Let $X$ be an irreducible three-dimensional normal affine
variety. Assume that there exists a locally transitive
$\mathrm{SL}(2)$-action on $X$. Then $X$ is toric if and only if
$X$ is an $\mathrm{SL}(2)/\mathbb{Z}_r$-embedding of height $\frac
pq$, where $(p,q)=1$ and $r$ is divisible by $q-p.$
\end{theor}
\begin{corollary}
An $\SL(2)$-embedding is toric if and only if its height can be
written as $p/(p+1).$
\end{corollary}

\section {The cone of a toric $\SL(2)$-embedding}

\label{cone}As it was mentioned in Section \ref{tor} any toric
variety $X$ corresponds to a polyhedral cone in the space
$N_{\mathbb{Q}}\cong\mathbb{Q}^{\dim X}$. Let us describe the
cone corresponding to an $\SL(2)/\mathbb{Z}_ {(q-p)l}$-embedding
$X$ of height $p/q$.

In the previous section we gave an explicit construction of $X$ as
the quotient of a toric $\SL(2)/\mathbb{Z}_{(q-p)}$-embedding $Y$
by the action of the group $\mathbb{Z}_l$. The algebra of regular
functions on $Y$ is $\mathbb{K}[x_1^mx_2^{q-m}y_1^ny_2^{p-n}]$,
$0\leq m\leq q$, $0\leq n\leq p$. Moreover, we obtained an
explicit formula for the action of $\mathbb{Z}_l$ on
$\mathbb{K}[Y]$. The action of the generator $\zeta$ of
$\mathbb{Z}_l$ on $\mathbb{K}[Y]$ is given by $\zeta\cdot
x_1^mx_2^{q-m}y_1^ny_2^{p-n}=
\varepsilon^{p-q}x_1^mx_2^{q-m}y_1^ny_2^{p-n},$ where
$\varepsilon$ is an $l(q-p)$-th root of unity. Hence the algebra
of invariants is
$$\mathbb{K}[X]=\mathbb{K}[Y]^{\mathbb{Z}_l}=
\mathbb{K}[x_1^ux_2^{lq-u}y_1^vy_2^{lp-v}].$$

Let us embed $\mathbb{K}[X]$ in the algebra of polynomials in
three variables. Suppose
$f=x_2/x_1,g=y_2/y_1,h={x_1}^{lq}{y_1}^{lp}$  in $\mathbb{K}(X)$.
It is clear that $f$, $g$ and $h$ are algebraically independent.
Then ${x_1}^{lq-u}{x_2}^m{y_1}^{lp-v}{y_2}^{v}=f^{u}g^{v}h$.
There is a natural embedding $\mathbb{K}[f,g,h]\hookrightarrow
\mathbb{K}[f,g,h,f^{-1},g^{-1},h^{-1}].$ It corresponds to the
embedding of the torus
$T=\mathrm{Spec}\,\mathbb{K}[f,g,h,f^{-1},g^{-1},h^{-1}]$ into the
three-dimensional affine space
$V=\mathrm{Spec}\,\mathbb{K}[f,g,h].$ Recall that $M$ is the
lattice of characters of $T$. Then $M$ is spanned by the vectors
$a=(1,0,0),b=(0,1,0)$ and $c=(0,0,1)$. Let $N$ be the dual
lattice. Then $\mathbb{K}[f,g,h]$ is the semigroup algebra  of
the semigroup $P$ generated by $a,b$ and $c$. Let us define an
isomorphism $i\colon \mathbb{K}[f,g,h]\rightarrow \mathbb{K}[P]$
by $i(f)=a,\ i(g)=b,\ i(h)=c.$ Then $i(f^{u}g^{v}h)=ua+vb+c.$

Denote by $\sigma$ the cone corresponding to $X$ and by
$\widehat{\sigma}\,$ the cone dual to $\sigma$. The cone
$\widehat{\sigma}$ is spanned by the vectors $$ua+vb+c,
u\in\mathbb{Z}\cap[0,lq],v\in \mathbb{Z}\cap[0,lp].$$ Then the
cone $\sigma$ consists of all vectors $w\in N$ such that $\langle
w,ua+vb+c\rangle\geq 0.$ A linear function accepts the minimal
value at an end of an interval. Hence the cone $\sigma$ can be
given by the following
inequalities:$$(w,c)\geq0,(w,lqa+c)\geq0,(w,lpb+c)\geq0,(w,lqa+lpb+c)\geq0.$$
To find edges of the cone $\sigma$ we should choose two of the
above four inequations, replace them by the corresponding
equations, and solve the obtained system of equations. Besides,
we need to choose only those solutions that belong to $\sigma$.
We have six systems.

\begin{trivlist}{}{}
\item{1)}
$$\left\{\begin{aligned}
t= 0;\\
rlq+t= 0.
\end{aligned}
\right.\qquad \parbox {8cm}{Hence $r=0,t=0$. The answer is
(0,1,0).}$$

\item{2)}
$$\left\{\begin{aligned}
t= 0;\\
slp+t= 0.
\end{aligned}
\right.\qquad\parbox{8cm}{Therefore $s=0,t=0$. The answer is
(1,0,0).}$$

\item{3)}
$$\left\{\begin{aligned}
t= 0;\\
rlq+slp+t= 0.
\end{aligned}
\right.\qquad \parbox{8cm}{Hence $t=0, rlq+slp=0$. Since
$rlq\geq0$ and $slp\geq0$, we obtain $rlq=slp=0$, that is,
$r=s=t=0$.}$$

\item{4)}
$$\left\{\begin{aligned}
rlq+t= 0;\\
slp+t= 0.
\end{aligned}
\right.\qquad\parbox{8cm}{Therefore $0\leq rlq+slp+t=-t\leq0$.
Hence $r=s=t=0$.}$$

\item{5)}
$$\left\{\begin{aligned}
rlq+t= 0;\\
rlq+slp+t= 0.
\end{aligned}
\right.\qquad \parbox{8cm} {Therefore $s=0$. Since for vector
$(1,0,-lq)$ we obtain $slp+t<0$, the answer is $(-1,0,lq)$.}$$

\item{6)}
$$\left\{\begin{aligned}
slp+t= 0;\\
rlq+slp+t= 0.
\end{aligned}
\right.\qquad \parbox{8cm}{Hence $r=0$. The answer is
$(0,-1,lp)$.}$$
\end{trivlist}

We obtain four 1-dimensional faces: $\rho_1=\mathbb{Q}_+(1,0,0)$,
$\rho_2=\mathbb{Q}_+(0,1,0)$, $\rho_3=\mathbb{Q}_+(-1,0,lq)$ and
$\rho_4=\mathbb{Q}_+(0,-1,lp)$. Let us formulate the result.
\begin{prop}
The cone corresponding to an
$\SL(2)/\mathbb{Z}_{l(q-p)}$-embedding of height $p/q$ regarded as
a toric variety is $$\mathrm{cone}((1,0,0), (0,1,0),
(-1,0,lq),(0,-1,lp)).$$
\end{prop}

\section{Final remarks}
Recently the question on degeneracy of an algebraic variety to a
toric one has been actively studied (see, for example, \cite{AB}).
In the case of an affine $\mathrm{SL}(2)/\mathbb{Z}_r$-embedding
the standard procedure of contraction of an action \cite{Stag}
gives us a toric variety. But if the starting variety was toric,
then the variety after contraction is never isomorphic to it.

Note that the result of this paper can be consider as the first
step towards describing the group of all (not equivariant)
automorphisms of an $\mathrm{SL}(2)/\mathbb{Z}_r$-embedding.
Indeed, the rank of this group has been calculated. For toric
$\mathrm{SL}(2)/\mathbb{Z}_r$-embeddings it is equal to 3. For
non-toric $\mathrm{SL}(2)/\mathbb{Z}_r$-embeddings it equals 2,
because a 2-dimensional torus acts on any
$\mathrm{SL}(2)/\mathbb{Z}_r$-embedding (see \cite[chapter 3, p.
4.8]{kr}).


\end{document}